\documentclass{amsart}
\usepackage{amssymb,amsmath,amsthm,longtable,color,verbatim}
\usepackage{graphicx}

\newcommand{\E}{{\mathbb E}}
\newcommand{\bS}{{\mathbb S}}
\renewcommand{\H}{{\mathbb H}}
\newcommand{\1}{{ \mathbf  1}}
\renewcommand{\P}{{\mathbb P}}
\newcommand{\R}{{\mathbb R}}

\renewcommand{\Re}{\mathrm{Re}\,}
\renewcommand{\Im}{\mathrm{Im}\,}

\def\br#1{\left(#1\right)}
\def\brb#1{\left[#1\right]}

\newcommand{\SLE}{\mathrm{SLE}}

\definecolor{red}{rgb}{1,0,0}

\theoremstyle{remark}
\newtheorem{rem}{Remark}
 
\theoremstyle{remark}

\theoremstyle{definition}

\theoremstyle{plain}
\newtheorem{theorem}{Theorem}[section]

\newtheorem{lemma}[theorem]{Lemma} 
\newtheorem{proposition}[theorem]{Proposition} 
\date{}
\title{A proof of factorization formula for critical percolation}
 
\author[D. Beliaev]{Dmitri Beliaev}
\email{dbeliaev@math.princeton.edu}
\address{Princeton University, Department of Mathematics, Fine Hall Washington Rd, Princeton NJ 08544}
\author[K. Izyurov]{Konstantin Izyurov}

\begin{document}
\begin{abstract}

We give mathematical proofs to a number of statements which appeared in the series of papers by Kleban, Simmons and Ziff \cite{SKZ,SZK09} where they computed the probabilities of several percolation events. 

\end{abstract}

\maketitle

\section{Introduction}

Two-dimensional critical percolation has been extensively studied in the last two decades, both from a Conformal Field Theory point of view and by means of Schramm-Loewner evolution. The first approach provides a way to obtain explicit formulas for correlation functions (connection probabilities), whereas the second one, in particular, gives a key to rigorous proof of such results. The proof of Cardy's formula was the major step in Smirnov's proof of conformal invariance of percolation on the triangular lattice \cite{Smirnov01}. Lawler, Schramm and Werner \cite{LSW02} proved an asymptotical formula for the probability of an interval on the boundary of the domain to be connected to a small neighborhood of a bulk point. Later, Watts' formula for the simultaneous occurence of up-down and left-right crossing in a rectangle was proven \cite{DubedatWatts}, \cite{SheffieldWatts}. A formula for the expected number of clusters separating two sides of a rectangle was established in \cite{HonSmiPerc}, although the proof did not involve SLE techniques. We refer the reader to \cite{Lawler05} for the introduction to the SLE topic.

In \cite{SZK09} Simmons et al. computed the density of the probability that there is a percolation cluster which connects a boundary interval with a boundary point and an interior point. In the present paper we provide a precise formulation and a proof of this result. This also leads to a rigorous proof of the exact factorization formula for certain four-point correlation function proposed by the same authors.

Throughout the paper by percolation we mean the scaling limit of the critical percolation on the triangular lattice. We work only with percolation interfaces (boundaries of clusters) which by \cite{Smirnov01} have the law of $SLE(6)$. We remark here that this is only proven for the cite percolation on triangular lattice, but is conjectured to be true for other lattices as well. 


Let $u_1<u_2<u_3$ be three points on the real line and $w$ be a point in the upper half-plane $\H$. We are going to compute the asymptotic behavior of the probability that the percolation cluster attached to a boundary interval $[u_1,u_2]$ approaches  small neighborhood of $u_3$ and $w$. We will understand those neighborhoods in the following sense which is well-suited for SLE computations. Let $K$ be the percolation cluster attached to $[u_1,u_2]$, by $r(w)$ we denote the conformal radius of the component of $\H \setminus K$ which contains $w$ as seen from $w$. For a boundary point $u_3$, we set $r(u_3):=dist(K\cap\R, u_3)$.

Denote by $\widetilde{F}(u_1,u_2,u_3,w,s_1,s_2)$ the probability that $r(w)<e^{-s_1}$ and $r(u_3)<e^{-s_2}$. We can write 
$$
\widetilde{F}(u_1,u_2,u_3,w,s_1,s_2)=C(u_1,u_2,u_3,s_2)F(u_1,u_2,u_3,w,s_1,s_2),
$$
where $C(u_1,u_2,u_3,s_2)$ is the probability that $r(u_3)<e^{-s_2}$ and $F(u_1,u_2,u_3,w,s_1,s_2)$ is the conditional probability that $r(w)<e^{-s_1}$ given that $r(u_3)<e^{-s_2}$ We also introduce the following notation:

\begin{equation}
\label{defC}
C(u_1,u_2,u_3):= \lim\limits_{t\rightarrow \infty} e^{\frac13 s}C(u_1,u_2,u_3,s);
\end{equation}
\begin{equation}
\label{defF1}
F(u_1,u_2,u_3,w,s_1):= \lim\limits_{s_2\rightarrow \infty} F(u_1,u_2,u_3,w,s_1,s_2);
\end{equation}
\begin{equation}
\label{defF2}
F(u_1,u_2,u_3,w):= \lim\limits_{s_1\rightarrow \infty} e^{\frac{5}{48} s_1}F(u_1,u_2,u_3,w,s_1);
\end{equation}

The existence of these limits will be clear from what follows. Denote $\bS=\{0<\Re z<\infty, 0<\Im z<1\}$. The main result of this paper is the following:

\begin{theorem}
\label{ThmF}
One has
\begin{equation}
\label{EqF}
F(u_1,u_2,u_3,w)=K_5|\psi'_{u_1,u_2,u_3}(w)|^{5/48}G(\Re\psi_{u_1,u_2,u_3}(w),\Im\psi_{u_1,u_2,u_3}(w)),
\end{equation}
where  $\psi_{u_1,u_2,u_3}$ is the conformal map that transforms $\{\H, u_1,u_2,u_3\}$ to $\{\bS, i,0,\infty\}$, $K_5$ is a constant given by (\ref{valueK}) and $G$ is an explicit function given by (\ref{F})
\end{theorem}

We emphasize that the function $\widetilde{F}(u_1,u_2,u_3,w,s_1,s_2)$ is a probability in \emph{continuous} percolation, that is, the scaling limit of corresponding probability for percolation on the triangular lattice. Hence $F(u_1,u_2,u_3,w)$ is a double limit: we first take the mesh size to zero for fixed values of $s_1$, $s_2$, and then take those to zero. A more natural formulation would concern the scaling behaviour of probabilities that a particular boundary site, interiour cite and a boundary segment are touched by the same cluster. We believe the result to be true in this formulation as well, but technical difficulties seem insuperable. 

Thanks to conformal invariance of critical percolation, Theorem \ref{ThmF} has the following corollary: if $\Omega$ is a simply-connected domain with marked points $u_1$, $u_2$, $u_3$ on the boundary, $w$ inside, and $\partial \Omega$ is smooth near $u_3$, then 
$$
F_{\Omega}(u_1,u_2,u_3,w)=|\varphi'(u_3)|^{1/3}|\varphi'(w)|^{5/48}F(\varphi(u_1),\varphi(u_2),\varphi(u_3),\varphi(w)),
$$ 
where $\varphi$ is a conformal map from $\Omega$ to $\H$ and $F_{\Omega}$ is the probability for the percolation in $\Omega$ that $u_3$ is connected to $(u_1,u_2)$ and $w$, defined by the same limiting procedure as $F$.

To prove Theorem \ref{ThmF}, we proceed as follows. In the section \ref{ThreeFour}, we recall the formulae for three- and four-point functions for percolation. In the section \ref{FivePt}, we prove that $SLE(6)$ started from $u_2$ and conditioned to hit $u_3$ and not to swallow $u_1$ has the law of $SLE(6,2,-2)$. Based on this, we write down certain PDE's for $F(u_1,u_2,u_3,w)$ and $FC$, the latter being the same as obtained in \cite{SZK09}. Finally, we prove that these PDE's together with boundary conditions determine $F$ uniquely. In the section \ref{Factorization} we use our result to prove a factorization formula (\ref{EqFact}) that expresses the probability density of two points on the boundary of the domain to be connected to a bulk point in terms of pairwise connection probabilities. 

For the next two sections, we will assume that all functions are sufficiently smooth, so that the forthcoming It\^o calculus is legal. Surpassing this assumption is quite standard: after getting an explicit answer, which turns out to be a smooth function, one plugs that answer into the previous computations to prove that it is indeed a martingale, and then applies optional stopping theorem. This is explained in detail for the function $F$ in the section \ref{FivePt}.

\section{Three- and Four-point functions}
\label{ThreeFour}
The function $C(u_1,u_2,u_3,s)$ is the probability that there is a percolation crossing connecting $[u_1,u_2]$ with $[u_3-e^{-s},u_3+e^{-s}]$, or, in terms of interfaces, the probability that the $SLE(6)$ curve started from $u_2$ touches the interval $[u_3-e^{-s},u_3+e^{-s}]$ before swallowing $u_1$. Hence $C(u_1,u_2,u_3,s)$ is given by Cardy's formula. Taking the limit as $s\rightarrow \infty$, one gets the following result:

\begin{lemma} 
\label{lemma3pt}
As $s\rightarrow\infty$, one has
$$
C(u_1,u_2,u_3,s)\sim K_3 e^{-s/3}\br{\frac{u_2-u_1}{(u_3-u_2)(u_3-u_1)}}^{1/3},\quad\text{where}
$$
$$
K_3=\frac{\sqrt{\pi}}{\Gamma(\frac13)\Gamma(\frac76)}.
$$
\end{lemma}
Hereinafter the notation $A \sim B$ means that the ratio of the two sides tends to one.
\begin{proof}
Theorem 3.2 in \cite{LSW1} provides this lemma when $u_1,u_2,u_3,e^{-s}$=$1,\infty,x/2,x/2$. The statement of the lemma is obtained by conformal mapping to this case.
\end{proof}

\begin{rem}
The rate of decay $e^{-s/3}$ of $C$ is consistent with the value $1/3$ of the boundary one-arm exponent. Note that our event is similar to the one-arm event, the difference is that we condition cluster to hit the interval $[u_1,u_2]$.
\end{rem}

\begin{rem}
One can check by direct computation that the function $C(u_1,u_2,u_3)$ from lemma \ref{lemma3pt} satisfies the following equation: 
\begin{equation}
\label{Cardy}
\br{\frac{2}{u_1-u_2}\partial_{u_1}+\frac{2}{u_3-u_2}\partial_{u_3}+3\partial^2_{u_2}-\frac{2}{3(u_3-u_2)^2}}C=0.
\end{equation}
This equation means that $C$ is a martingale which is conformally $1/3$-covariant at $u_3$ and invariant at $u_1$ and $u_2$, that is, $|g'_t(u_3)|^{1/3}C(g_t(u_1),B_{6t},g_t(u_3))$ is a martingale, where $g_t$ is the Loewner map driven by the Brownian motion $B_{6t}$ started from $u_2$.
\end{rem}

Now we would like to prove an analog of this lemma when the point $u_3$ is inside the domain. Let  $F(u_1,u_2,w,s)$ be the probability that $r(w)<e^{-s}$.
\begin{lemma} 
\label{lemma4pt}
As $s\rightarrow\infty$, one has
\begin{equation}
\label{FourPt}
F(u_1,u_2,w,s)\sim K_4 e^{-\frac{5}{48}s}|\phi'(w)|^{5/48}\br{\sin(\pi \omega/2)}^{1/3},
\end{equation}
where $\omega$ is the harmonic measure of $(u_1,u_2)$ seen from $w$; $\phi$ is a conformal map from $\H$ to the unit disc, $\phi(w)=0$, and
$$
K_4=\frac{18}{5\pi}.
$$
\end{lemma}
\begin{proof}
Lawler, Schramm and Werner \cite{LSW02} proved a slightly weaker statement. Here we refine their result using spectral theory techniques. 

Using locality of SLE(6), we reformulate the problem in terms of radial SLE(6). Let $h(\theta,s)$ denote the probability that the conformal radius of the complement of percolation cluster in the unit disc  attached to a boundary arc of length $\theta$ is less than $e^{-s}$. Then it is proven in \cite{LSW02} that 
\begin{equation}
\label{fourPtweak}
c_1(\sin(\theta/4))^{1/3}e^{-5/48 s}\leq h(\theta,s)\leq c_2(\sin(\theta/4))^{1/3}e^{-5/48 s}
\end{equation}
and that the function $h$ satisfies the PDE:
\begin{equation}
\label{hPDE}
\partial_s h = (3\partial_{\theta\theta}+\cot(\theta/2)\partial_\theta)h
\end{equation}
with boundary conditions $h(\theta,0)\equiv 1$, $h(0,s)\equiv 0$, and (a weak form of) Neumann boundary condition ``$\partial_\theta h(\theta,s)\equiv 0$'' at $\theta=2\pi$. Write $f(\theta,s):=\left(\sin(\theta/2)\right)^{1/3}h(\theta, s)$. Then (\ref{hPDE}) implies
\begin{equation}
\label{fPDE}
\partial_s f = (3\partial_{\theta\theta}+3V)f,
\end{equation}
where $V(\theta)=\frac{1}{12}+\frac{1}{18}\cot^2(\frac{\theta}{2})$. Our goal is to define a self-adjoint operator corresponding to the right-hand side of (\ref{fPDE}) and to prove that it has a discrete spectrum, hence the function $f$ for large $s$ behaves like the leading eigenfunction. The proof below (and perhaps the result itself) is very standard in spectral theory. First of all, we consider the operator $\partial_{\theta\theta}+V$ acting on smooth functions with support inside $(0,2\pi)$ and write explicitly the domain $D(\Lambda_0)$ of the $L_2$-closure of this operator. Self-adjoint extensions of this closure are obtained by adding two-dimensional subspaces to $D(\Lambda_0)$. Different extensions correspond to different boundary conditions, and we consider two such extensions  $\Lambda$ (which corresponds to the boundary conditions we are interested in) and  $\Lambda_1$ (which corresponds to Dirichlet boundary conditions for $h$ on both sides of $(0,2\pi)$). It is a well-known fact that the spectral properties of $\Lambda$ and $\Lambda_1$ are similar, so it is sufficient to prove that $\Lambda_1$ has a discrete spectrum. We reformulate this property in terms of quadratic forms, and then use Hardy's inequality to prove that the form corresponding to $\Lambda_1$ is comparable to the form corresponding to $\partial_{\theta\theta}$, which is well-known to have a discrete spectrum.

Set $\Lambda_0:=(\partial_{\theta\theta}+V)$ with 
\begin{eqnarray*}
D(\Lambda_0)&:=& 
\\
&=& \{f\in AC^1([0,2\pi]):\Lambda_0 f\in L_2(0,2\pi),f(0)=f(2\pi)=f'(0)=f'(2\pi)=0\}
\\
&= &\{f\in AC^1([0,2\pi]):f''\in L_2(0,2\pi),f(0)=f(2\pi)=f'(0)=f'(2\pi)=0\}
\end{eqnarray*}

The operator $\Lambda_0$ is closed in $L_2$. Let $\xi$ be a smooth decreasing function with $\xi\equiv 1$ on $[0,\pi/3)$ and $\xi\equiv 0$ on $(5\pi/3,2\pi]$. We define the functions 
$$
w_1(x):=(2\pi-x)^{1/3}\xi(2\pi-x);\quad w_2(x):=x^{2/3}\xi(x);\quad w_3(x):=(2\pi - x)^{2/3}\xi(2\pi-x)
$$
and the extensions $\Lambda$, $\Lambda_1$ of $\Lambda_0$ with
$$
D(\Lambda)=D(\Lambda_0)+\langle w_1\rangle +\langle w_2\rangle;\quad D(\Lambda_1)=D(\Lambda_0)+\langle w_2\rangle+\langle w_3\rangle
$$
By Theorems 2 and 4, \S 18 of \cite{Naimark}, $\Lambda$ and $\Lambda_1$ are self-adjoint extensions of $\Lambda_0$. Note that $\Lambda_1$ corresponds to Dirichlet boundary conditions for the initial PDE, whereas $\Lambda$ corresponds to Dirichlet boundary condition at $0$ and Neumann boundary condition at $2\pi$.

We now prove that $\Lambda_1$ has discrete spectrum. It is sufficient to show that for any $A\in \R$, any subspace $L\subset D(\Lambda_1)$ such that all $g\in L$ obey the inequality
\begin{equation}
\label{Spec1}
A(g,g)<(\Lambda_1 g;g)
\end{equation}
is finitely dimensional. The potential $V(\theta)$ has singularities at $0$, $2\pi$ of order $\frac{2}{9}\theta^{-2}$ and $\frac{2}{9}(2\pi-\theta)^{-2}$ correspondingly. By Hardy's inequality, for any $\alpha>\frac89$ there exists a constant $B$ such that
\begin{equation}
\label{Forms}
(Vg,g)\leq \alpha (g',g')+B(g,g).
\end{equation}
for any $g\in W_2^1((0,2\pi))$ such that $g(0)=g(2\pi)=0$, and hence for any $g\in D(\Lambda_1)$. Therefore (\ref{Spec1}) implies 
$$
\frac{A-B}{1-\alpha}(g,g)\leq-(g',g').
$$
Since the spectrum of $g\mapsto g''$ defined on $g\in W_2^1((0,2\pi))$, $g(0)=g(2\pi)=0$ is discrete and negative, this cone can only contain finitely dimensional subspaces. 

By Theorem 2, \S 19 of \cite{Naimark}, this implies that the spectrum of $\Lambda$ is also discrete. Standard arguments now imply that all eigenvalues of $\Lambda$ are simple, and that the positive eigenfunction $\psi_0:=(\sin(\theta/2)\sin(\theta/4))^{1/3}$ of $\Lambda$ corresponds to the largest eigenvalue $-5/144$. It has been shown in \cite{LSW02} that for $\varepsilon =1$ and any $s>0$,
$$
h^{\varepsilon}(\theta,s):=\int_0^\varepsilon h(\theta,s+s')ds'
$$
satisfies Neumann boundary condition ($\partial_{\theta}h^{\varepsilon}(2\pi,s)=0$) at $2\pi$ and Dirichlet boundary conditions at $0$. First, we observe that the same is true (and with the same proof) for any $\varepsilon>0$. Second, the functions $h^{\varepsilon}$ are monotone decreasing in $s$ and increasing in $\theta$. Hence (\ref{hPDE}) implies that $h^{\varepsilon}(\cdot,s)$ is concave for any $s$. By some elementary calculus, one deduces from this and (\ref{fourPtweak}) that $\partial_\theta h^{\varepsilon}\leq C\theta^{-2/3}$ as $\theta\rightarrow 0$. Observe that also  $\lim_{\theta\rightarrow 2\pi} \partial_{\theta}h=0$.

We now consider $f^{\varepsilon}(\theta,s)=(\sin(\theta/2))^{1/3}h^{\varepsilon}(\theta,s)$ and write
$$
f^{\varepsilon}(\cdot,s)=\sum c_k(s)\psi_k(\cdot), 
$$
where $\psi_k$ are eigenfunctions of $\Lambda$ and the equality holds in the sence of $L_2$. By multiplying (\ref{fPDE}) by $\psi_k$ and integrating by parts, taking into account the boundary conditions for $h$ derived above, we obtain
$$
\partial_t c_k(s)=\lambda_kc_k(s),
$$
and hence 
$$
f^{\varepsilon}(\cdot,s)e^{5/48s}\stackrel{L_2}{\longrightarrow}C_\varepsilon \psi_0 \quad \text{as} \quad s\rightarrow\infty
$$
Since the functions $h_\varepsilon(\cdot,s)$ are concave for all $s$, this actually implies uniform convergence. Corresponding result for $h$ is obtained by taking $\varepsilon\rightarrow 0$, and the constant $K_4$ is computed by projecting the initial conditions $h(\cdot,0)\equiv 1$ to the main eigenspace.
\end{proof}

\section{Five-point function}
\label{FivePt}
To compute $F$ we have to consider the law of SLE(6) curve started from $u_2$ and conditioned on the event $\mathcal{E}_s$ that $r(u_3)<e^{-s}$. Recall that in terms of interfaces, the event $\mathcal{E}_s$ means that the curve hits $(u_3-e^{-s};u_3+e^{-s})$ before swallowing $u_1$, and $\P[\mathcal{E}_s]=C(u_1,u_2,u_3,s)$. Let $u_1(t)$, $u_2(t)$, $u_3(t)$ and $r(u_3)$ be defined on a filtered probability space $(\Omega,{\mathcal F}_t,\P)$, $u_2(t)=B_{6t}$ and $u_{1,3}(t)=g_t(u_{1,3})$, where $B$ is a standard Brownian motion under $\P$ started from $u_2$ and $g_t$ is the Loewner map driven by $u_2(t)$. Let $\widetilde \P$ be the measure $\P$ conditioned on $\mathcal{E}_s$. We study the law of the driving force with respect to $\widetilde \P$. We have
$$
\frac{d \widetilde \P}{d\P}|_t=\frac{C(u_1(t),u_2(t),u_3(t),s(t))\1_{t\le T(u_3)}+\1_{t>T(u_3), r(u_3)<e^{-s}}}{C(u_1,u_2,u_3,s)}=D_t,
$$
where $T(u_3)$ is the swallowing time for $u_3$.
By Girsanov's theorem, $\widetilde B_t$ is a standard Brownian motion under $\widetilde \P$ if
$$
d\widetilde B_t=d B_t-\frac{\left\langle B_t,D_t \right\rangle}{D_t}. 
$$
Standard argument now shows that $D_t$ is a martingale, and 
$$
d D_t=\frac{\partial_{u_2} C(u_1(t),u_2(t),u_3(t),s(t))}{C(u_1(t),u_2(t),u_3(t),s(t))}\sqrt{6}dB_t
$$  
which implies that
$$
d B_t=d \widetilde B_t +\sqrt{6}\frac{\partial_{u_2} C(u_1(t),u_2(t),u_3(t),s(t))}{C(u_1(t),u_2(t),u_3(t),s(t))}dt.
$$
This proves that under condition $r(u_3)<e^{-s}$ the driving force of the Loewner evolution becomes
$$
d u_2(t)=\sqrt{6}d\widetilde B_t+6\frac{\partial_{u_2} C(u_1(t),u_2(t),u_3(t),s(t))}{C(u_1(t),u_2(t),u_3(t),s(t))}dt.
$$
The direct computation shows that in the limit, as $s\rightarrow \infty$, one has
$$
du_2=\sqrt{6}\widetilde B_t + \br{\frac{-2}{u_2-u_3}+\frac{2}{u_2-u_1}}dt
$$
which is the driving force of $\SLE(6,2,-2)$ (see \cite{DubedatSLEmart}, where these processes were introduced) started from $u_2$, $u_1$ and $u_3$. We would like to point out that the force point with $\rho=2=6-4$ conditions SLE(6) not to swallow the force point and $\rho=-2=6-8$ conditions SLE(6) to hit the force point. 
\begin{rem}
For general $\kappa$, the above argument does not work: $\mathcal{E}_s$ is now a five-point event and its probability is not expressed in terms of the four-point Cardy's formula. By conditioning first on hitting $u_3$, and then on not swallowing $u_1$, it is still possible to compute the drift in the limit $s\rightarrow\infty$; this leads to 
$$
du_2=\sqrt{\kappa}\widetilde B_t + \br{\frac{\kappa-8}{u_2-u_3}+\frac{\kappa G'(\frac{u_2-u_1}{u_3-u_1})}{(u_3-u_1)G(\frac{u_2-u_1}{u_3-u_1})}}dt,
$$
where $G(x)=\int_0^x\theta^{-\frac{4}{\kappa}}(1-\theta)^{\frac{2(6-\kappa)}{\kappa}}d\theta$. Hence for $\kappa\neq 6$ conditioning on both events is not the same as adding corresponding force points. 
\end{rem}
Let  $\tau$ be the stopping time which is the minimum of $T(w)$ -- the swallowing time for $w$ and $T(u_3)$ -- the swallowing time for $u_3$. By $\Omega$ we denote the component of the complement of the SLE trace up to time $\tau$ which contains the point $w$.

For a particular realization $\gamma([0,\tau])$ of conditioned SLE interface, the probability $F$ is the probability that the conformal radius of $\Omega$ in the complement of the cluster which is attached to the union of the left side of $\gamma$ and $[u_1,u_2]$ is at most $e^{-s}$. Let $E$ be the intersection of the left side of $\gamma$ and $[u_1,u_2]$ with the boundary of $\Omega$, and let $\phi$ be the conformal map from $\Omega$ onto the unit disc which maps $w$ to the origin. By conformal invariance, the probability of our event is the same as the probability that in the unit disc the conformal radius about the origin of the complement of the percolation cluster which is attached to $\phi(E)$ is at most $e^{-s}|\phi'(w)|$. By lemma \ref{lemma4pt}, this probability behaves as 
$$
K_4e^{-s \frac{5}{48}}|\phi'(w)|^{5/48} \br{\sin(\pi \omega/2)}^{1/3},
$$
where $\omega$ is the harmonic measure of $E$ seen from $w$.
This proves that after factoring out $e^{-s5/48}$ and passing to the limit we have 
$$
F(u_1,u_2,u_3,w)= K_4\E\brb{|\phi'(w)|^{5/48}\br{\sin(\pi \omega/2)}^{1/3}},
$$
where the expectation is taken with respect to the law of $\gamma([0,\tau])$.

\begin{lemma}
The function $F$ is a solution of the following PDE
\begin{equation}\label{F_equation}
\begin{aligned}
\Lambda_F F=\frac{5}{96}\frac{-2}{(w-u_2)^2}F+\frac{5}{96}\frac{-2}{(\bar w-u_2)^2}F+\frac{2}{u_1-u_2}\partial_{u_1}F+
\frac{2}{u_3-u_2}\partial_{u_3}F+\\
3\partial{u_2}^2 F +6\frac{\partial_{u_2}C}{C}\partial_{u_2}F+\frac{2}{w-u_2}\partial_w F+\frac{2}{\bar w-u_2}\partial_{\bar w} F=0
\end{aligned}
\end{equation}\end{lemma}
\begin{proof}
As we proved before, $\gamma$ has a law of $\SLE(6;-2,2)$ trace. The driving function of this SLE is given by the solution of the SDEs
\begin{eqnarray*}
du_2(t)&=&\sqrt{6}d\widetilde B_t+\br{\frac{-2}{u_2(t)-u_3(t)}+\frac{2}{u_2(t)-u_1(t)}}dt=\sqrt{6}d\widetilde B_t+6\frac{\partial_{u_2}C}{C}dt\\
du_1(t)&=&\frac{2dt}{u_1(t)-u_2(t)}\\
du_3(t)&=&\frac{2dt}{u_3(t)-u_2(t)}.
\end{eqnarray*}
The domain Markov property implies that
$$
|g_t'(w)|^{5/48}F(u_1(t),u_2(t),u_3(t),w(t))
$$
is equal to $K_4\E[|\phi'(w)|^{5/48}\br{\sin(\pi \omega/2)}^{1/3}|{\mathcal F}_t] $ and hence is a martingale (here $w(t)=g_t(w)$ is the image of $w$ under the SLE map $g_t$). It is more convenient to write the It\^o's formula with respect to $w$ and $\bar w$ instead of $\Re w$ and $\Im w$. In these terms
$$
g_t'(w)^{5/96}g_t'(\bar w)^{5/96}F(u_1(t),u_2(t),u_3(t),w(t),\bar w(t))
$$
is a martingale. By It\^o's formula the drift is equal to zero, which is equivalent to equation \eqref{F_equation}.

\end{proof}

\begin{theorem}
\label{FCmartingale}
The product of two functions $F$ and $C$ is annihilated by the following operator
\begin{equation}
\begin{aligned}
\label{FCequation}
\Lambda=-\frac{5}{48}\frac{1}{(w-u_2)^2}-\frac{5}{48}\frac{1}{(\bar w-u_2)^2}-\frac{2}{3}\frac{1}{u_3-u_2}+\frac{2}{u_1-u_2}\partial_{u_1}+\\
\frac{2}{u_3-u_2}\partial_{u_3}+
3\partial{u_2}^2  +\frac{2}{w-u_2}\partial_w +\frac{2}{\bar w-u_2}\partial_{\bar w} 
\end{aligned}
\end{equation}
\end{theorem}
\begin{proof}
If we apply $\Lambda$ to $FC$ then the result will be
$$
\br{\Lambda F} C + F \br{\frac{2}{u_1-u_2}\partial_{u_1}C+
\frac{2}{u_3-u_2}\partial_{u_3}C+
3\partial{u_2}^2 C} + 6 \partial_{u_2}F \partial_{u_2}C.
$$
Using that 
$$
\Lambda=\Lambda_F-6\frac{\partial_{u_2}C}{C}\partial_{u_2}-\frac{2}{3}\frac{1}{u_3-u_2}
$$
the formula above can be rewritten as
$$
\begin{aligned}
F\br{-\frac{2}{3}\frac{1}{u_3-u_2}C+
\frac{2}{u_1-u_2}\partial_{u_1}C+
\frac{2}{u_3-u_2}\partial_{u_3}C+
3\partial{u_2}^2 C} 
\end{aligned}
$$
which is equal to zero by \eqref{Cardy}.
\end{proof}

The equation \eqref{FCequation} is exactly the equation (A.5) from \cite{SZK09}. 
Following \cite{SZK09}
we can transfer this equation into the semi-infinite strip (keeping covariance in mind) where after the right ansatz one can separate variables and find a solution to the equation (\ref{FCequation}):
$$
G(x,y):=
$$
\begin{equation}
\label{F}
\sinh(\pi x)^{-1/3}e^{\pi x/3}\,_2 F_1\br{-\frac{1}{2},-\frac{1}{3},\frac{7}{6},e^{-2\pi x}}\br{\frac{\sinh(\pi x)^2\sin(\pi y)^2}{\sinh(\pi x)^2+\sin(\pi y)^2}}^{11/96},
\end{equation}
where $x$ and $y$ are real and imaginary parts of the image of $w$ under the conformal transformation which maps the half-plane with marked point $u_1$, $u_2$ and $u_3$ onto the semi infinite strip $\bS=\{0<x<\infty, 0<y<1\}$.

This solution to $\eqref{FCequation}$ is not unique. 
In order to prove that $G$ indeed gives the correct formula for the right-hand side of (\ref{EqF}), we proceed as follows. Let $\psi_{u_1,u_2,u_3}$ (or $\psi$ for short) be the conformal map that transforms $\H, u_1,u_2,u_3$ to $\bS, i,0,\infty$. If we repeat the above It\^o computations with 
$$
H(u_1,u_2,u_3,w):=K_5|\psi'_{u_1,u_2,u_3}(w)|^{5/48}G(\Re\psi_{u_1,u_2,u_3}(w),\Im\psi_{u_1,u_2,u_3}(w))
$$ 
instead of $F$, we find that 
$$
|g_t'(w)|^{5/48}H(u_1(t),u_2(t),u_3(t),w(t))
$$
is a local martingale for the $SLE(6,2,-2)$ process.  We should now check that (for a right choice of the constant $K_5$) it has a correct value at the stopping time $\tau$, that is  
$$
\lim\limits_{t\rightarrow\tau}|g_t'(w)|^{5/48}F(u_1(t),u_2(t),u_3(t),w(t))=
$$
\begin{equation}
\label{Bdry_2}
\lim\limits_{t\rightarrow\tau}|g_t'(w)|^{5/48}H(u_1(t),u_2(t),u_3(t),w(t))
\end{equation}
Since $H$ is an explicit smooth function, all the It\^o computations are justified. We will finish the proof by checking uniform integrability and applying optional stopping theorem to $|g_t'(w)|^{5/48}(H-F)$ to conclude that it is identically zero.
 
There are three possible stopping scenarios (see Fig. 1). Consider them separately. In all cases $\phi_t$ denotes a conformal map from $\H\backslash \gamma_t$ to the unit disc with $\phi_t(w)=0$.

{\bf Case 1:} The curve swallows $w$ before hitting $u_3$ (i.e. $\tau=T(w)$) and does that by closing a clockwise loop. In this case the left-hand side of (\ref{Bdry_2}) tends to zero. However, the limit $\lim\limits_{t\rightarrow \tau}|\phi_t'(w)|$ exists and is finite almost surely. We can write $\phi_t=:h_t\circ g_t$, where $h_t(z)=\frac{w_t-z}{z-\overline{w_t}}$. It is hence sufficient to show that 
$$
|h_t'(w(t))|^{-5/48}H(u_1(t),u_2(t),u_3(t),w(t))=(2\Im w_t)^{5/48}H(u_1(t),u_2(t),u_3(t),w(t))
$$
tends to zero as $w(t)$ approaches $u_2$ in such a way that the harmonic measure of $(u_1,u_2)$ seen from $w(t)$ tends to zero. Without loss of generality, we can assume that $u_1,u_2,u_3=0,1,\infty$. Then 
\begin{equation}
\label{fromStrip}
(\psi_1)^{-1}:=(\psi_{0,1,\infty})^{-1}=\frac{\cosh(\pi z)+1}{2};
\end{equation}
and we have the following expansion around 1: 
\begin{equation}
\label{psi_1_p}
|\psi'_1(z)|=\frac{|z-1|^{-1/2}}{\pi}+o(|z-1|^{-1/2}), |z|\rightarrow 1.
\end{equation}
It remains to plug everything into the definition of $H$, taking into account that $|\Re\psi(w(t))|\gg|\Im\psi(w(t))|$, and we are done.
\begin{figure}[t]
\includegraphics[width=0.32\textwidth, trim=3cm 15cm 3cm 0cm, clip=true]{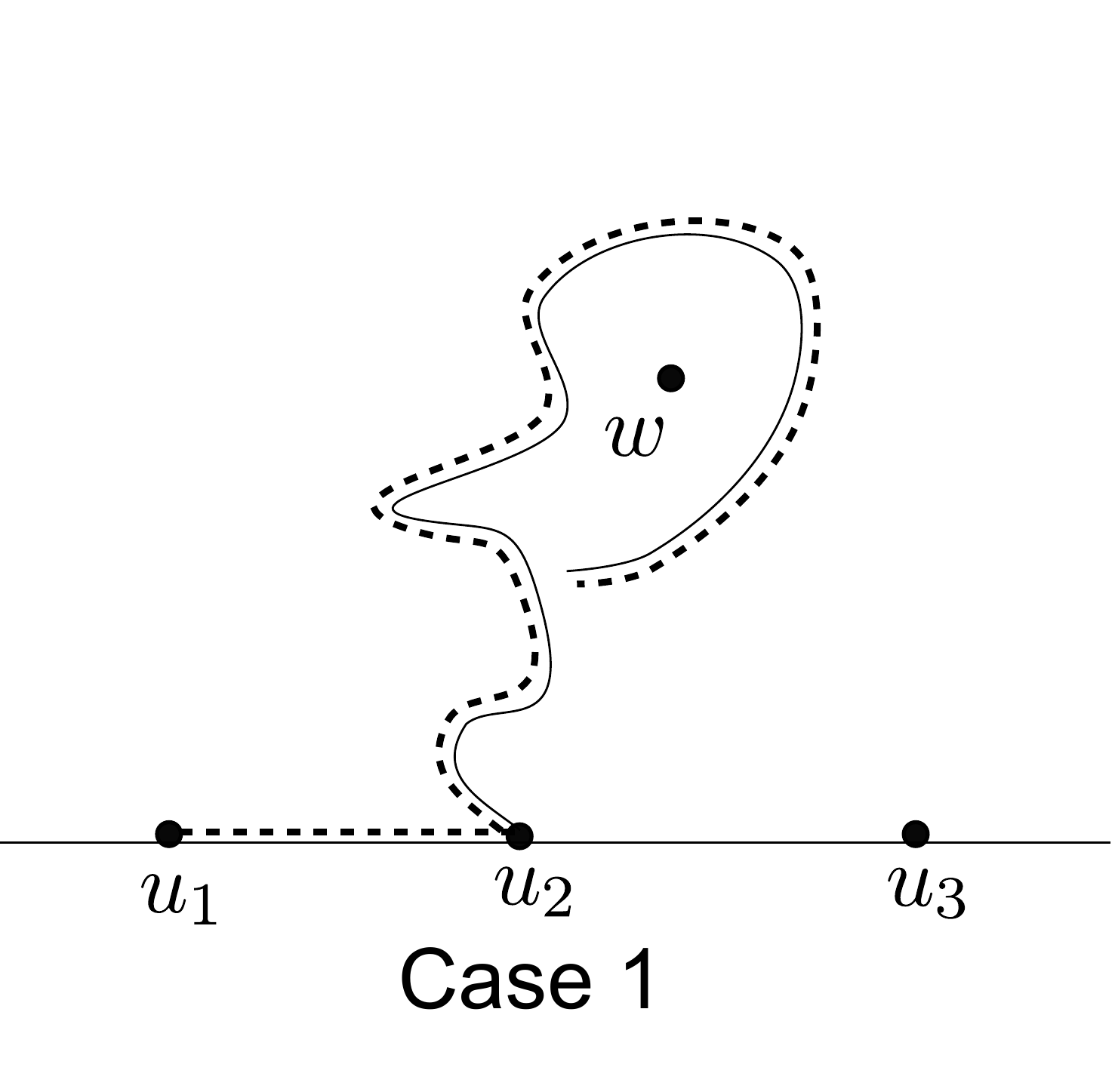}
\includegraphics[width=0.32\textwidth, trim=3cm 15cm 3cm 0cm, clip=true]{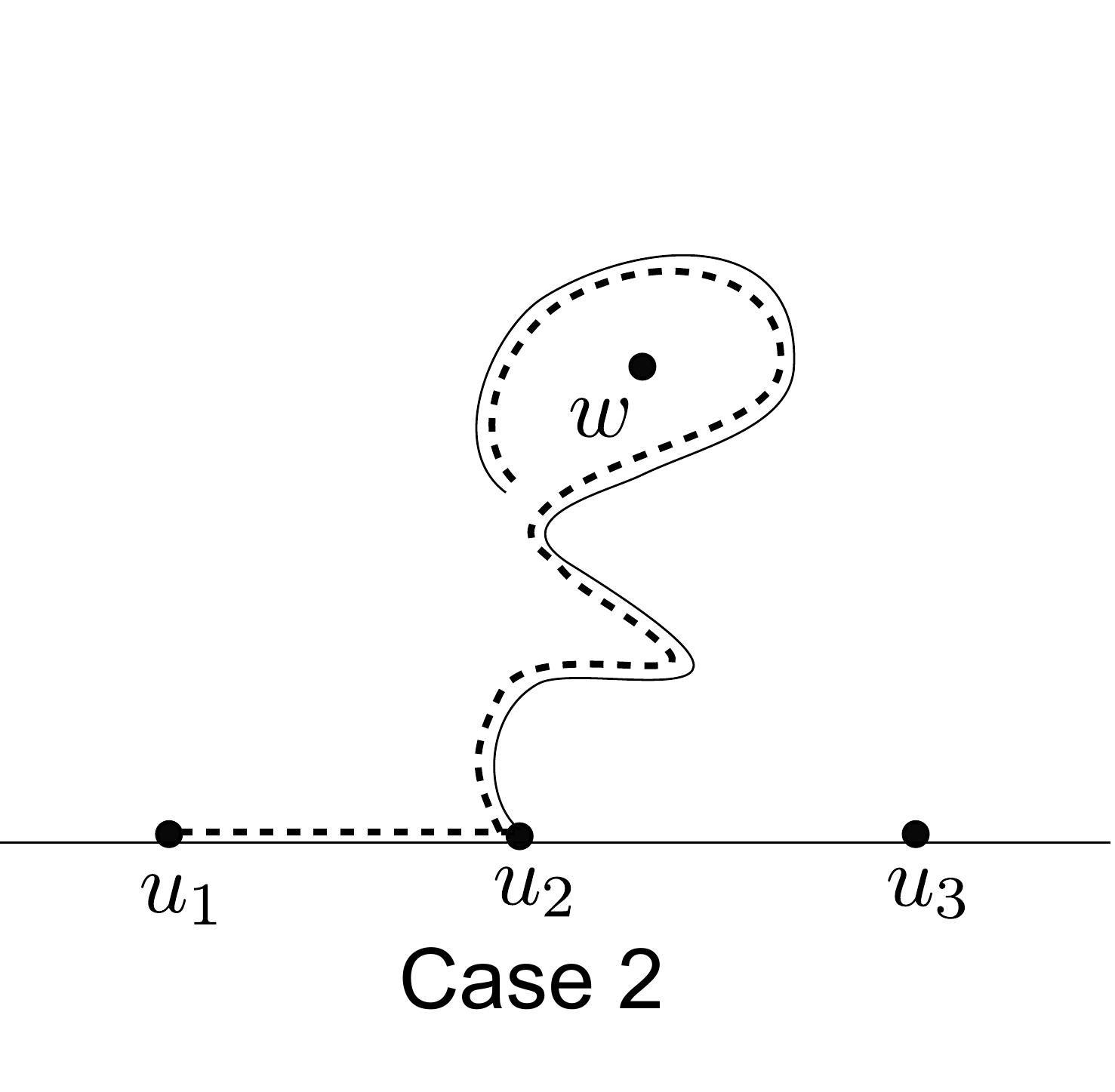}
\includegraphics[width=0.32\textwidth, trim=3cm 15cm 3cm 0cm, clip=true]{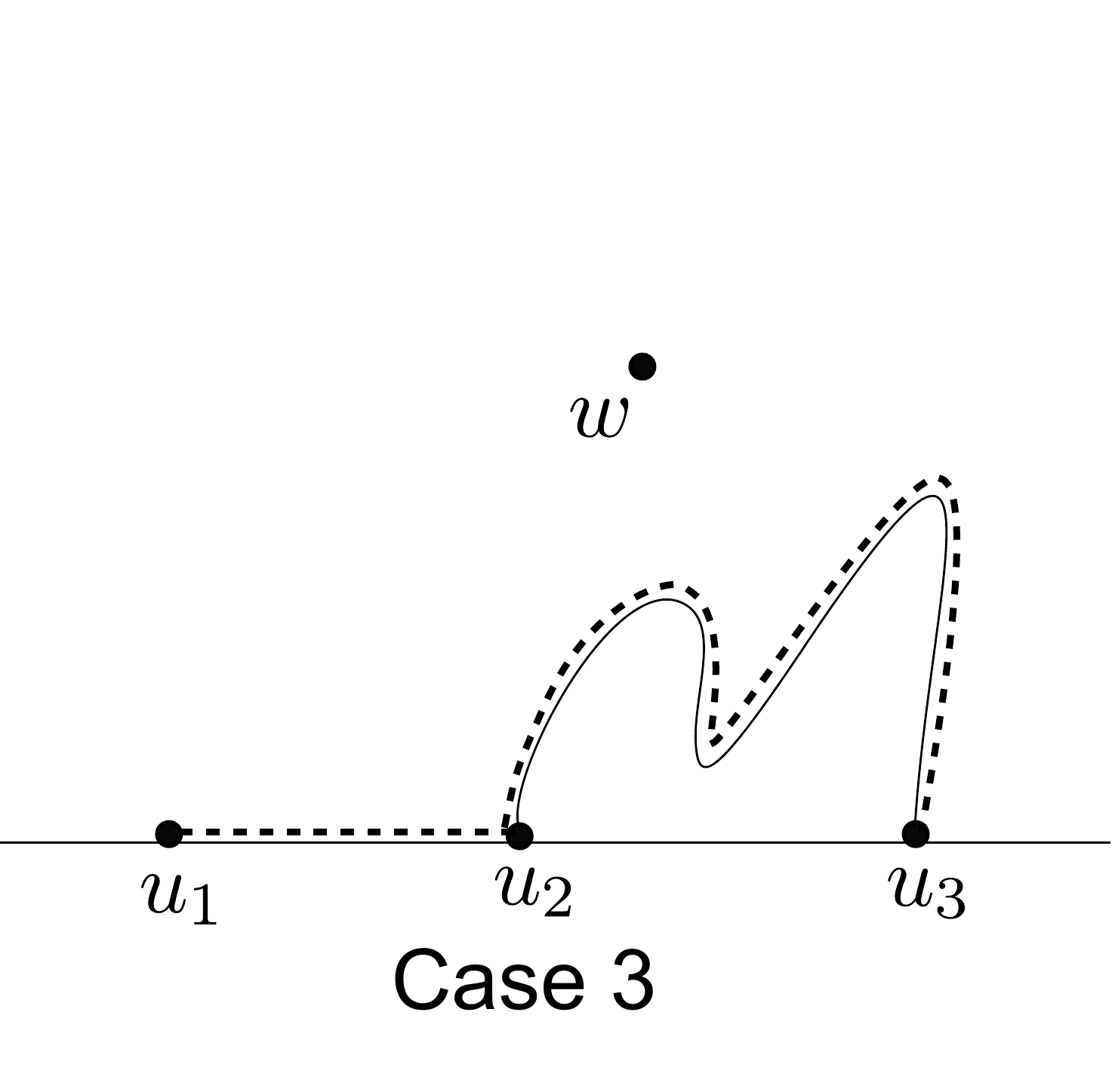}
\caption{Three possible stopping scenarios for the SLE(6,2,-2) interface. {Case 1:} at $t=\tau$ the curve cuts the point $w$ out of the cluster $K$ connected to ($u_1,u_2$). {Case 2:} at $t=\tau$, the whole boundary, as seen from $w$, belongs to $K$. {Case 3:} only the dashed piece of the boundary is explored to be a part of $K$.}
\end{figure}

{\bf Case 2:} The curve swallows $w$ before hitting $u_3$ (i.e. $\tau=T(w)$) and does that by closing a counterclockwise loop. In this case, the left-hand side of (\ref{Bdry_2}) tends to $K_4|\phi'(w)|^{5/48}$. Hence, it is sufficient to prove that 
$$
F(u_1(t),u_2(t),u_3(t),w(t))\sim H(u_1(t),u_2(t),u_3(t),w(t))
$$
as $w(t)$ approaches $u_2$ in such a way that the harmonic mearure of $(u_1,u_2)$ seen from $w(t)$ tends to 1 (that is, $|\Re\psi(w(t))|\ll|\Im\psi(w(t))|$). One has
$$
F(u_1(t),u_2(t),u_3(t),w(t))\sim K_4|h_t'(w(t))|^{5/48}\sim K_4(2\Im w_t)^{-5/48}.
$$
Once again, we can assume without loss of generality that $u_1,u_2,u_3=0,1,\infty$. Then, denoting $x(t)=\Re\psi(w(t))$, $y(t)=\Im\psi(w(t))$, one has
$$
G(x(t),y(t))\sim \,_2 F_1\br{-\frac{1}{2},-\frac{1}{3},\frac{7}{6},1} (\pi x(t))^{-5/48}
$$
and 
$$
|\psi'(w(t))|\sim\frac{|w(t)-1|^{-1/2}}{\pi}\sim\frac{2}{\pi^2}|y(t)|^{-1}. 
$$
Hence 
$$
H(u_1(t),u_2(t),u_3(t),w(t))\sim K_5\,_2 F_1\br{-\frac{1}{2},-\frac{1}{3},\frac{7}{6},1} \left(\frac{2}{\pi^{3}x(t)y(t)}\right)^{5/48}
$$
$$
\sim K_5\,_2 F_1\br{-\frac{1}{2},-\frac{1}{3},\frac{7}{6},1} \left(\frac{2}{\pi}\right)^{5/48}\frac{1}{(2\Im w(t))^{5/48}}
$$

{\bf Case 3: } The curve hits $u_3$ before swalloing $w$ (i.e. $\tau=T(u_3)$). In this case, we should compare the limits of $H$ and $F$ when $u_3\rightarrow u_2$. For $F$, answer is given by the lemma \ref{lemma4pt}
$$
\lim\limits_{u_3\rightarrow u_2}F(u_1,u_2,u_3,w)=
$$
\begin{equation}
\label{FourPt1}
K_4|\phi'(w)|^{5/48}\br{\sin(\pi \omega/2)}^{1/3}=\frac{K_4}{(2\Im w)^{5/48}}\br{\sin(\pi \omega/2)}^{1/3}
\end{equation}
where $\phi$ is a conformal map from the upper half-plane to the unit disc that maps $w$ to the origin, and $\omega$ is the harmonic measure of $(u_1,u_2)$ seen from $w$.

Consider now the limit of $H$. Note that as $u_3\rightarrow u_2$, the image $\psi(w)$ tends to $i$ along some direction. We write $\psi=\psi_1\circ\psi_2$, where $\psi_2$ maps $\H, u_1,u_2,u_3$ to $\H, 0,1,\infty$ and $\psi_1$ maps $\H, 0,1,\infty$ to $\bS, i,0,\infty$. It is convenient to choose $r=|\psi_2(w)|$ as a small parameter. Taking (\ref{fromStrip}) and (\ref{psi_1_p}) into account, one gets, as $u_3\rightarrow u_2$: 
$$
\begin{aligned}
\psi_2(w) = \frac{w-u_1}{w-u_3}\frac{u_2-u_3}{u_2-u_1}=:re^{i\theta}; \\
|\psi_2'(w)| = r\frac{|u_1-u_3|}{|w-u_1||w-u_3|}; \\
x:=\Re \psi(w)=\frac{2}{\pi}r^{\frac{1}{2}}\sin\frac{\theta}{2}+o(r^{\frac{1}{2}})\\
y:=\Im \psi(w)=1-\frac{2}{\pi}r^{\frac{1}{2}}\cos\frac{\theta}{2}+o(r^{\frac{1}{2}})\\
|\psi_1'(\psi_2(w))|=\frac{r^{-\frac{1}{2}}}{\pi}+o(r^{-\frac{1}{2}})\\
\end{aligned}
$$
Plugging everything to the definition of $H$, we find that 
$$
\lim\limits_{u_3\rightarrow u_2}H =K_5\frac{2^{-1/3}}{\pi^{5/48}}\,_2 F_1\br{-\frac{1}{2},-\frac{1}{3},\frac{7}{6},1}\sin(\theta/2)^{-\frac13}\sin(\theta)^{11/48} \left(\frac{|u_1-u_2|}{|w-u_1||w-u_2|}\right)^{5/48}.
$$
Note that $1-\theta/\pi$ is actually equal to the harmonic measure of $u_1,u_2$ seen from $w$. Simple computation now shows that 
$$
\frac{|u_1-u_2|}{|w-u_1||w-u_2|}=\frac{\sin(\theta)}{\Im w}, 
$$
hence
$$
\lim\limits_{u_3\rightarrow u_2}H =K_5\frac{2^{5/48}}{\pi^{5/48}} \,_2 F_1\br{-\frac{1}{2},-\frac{1}{3},\frac{7}{6},1} \frac{(\cos{\theta/2})^{1/3}}{(2\Im w)^{5/48}}.
$$

Combining the results of all three cases, we see that the equation (\ref{Bdry_2}) is satisfied if we choose
\begin{equation}
\label{valueK}
K_5^{-1}=\frac{2^{5/48}}{K_4\pi^{5/48}} \,_2 F_1\br{-\frac{1}{2},-\frac{1}{3},\frac{7}{6},1} 
\end{equation}
The only remaining part is applicability of the optional stopping theorem to the local martingale $|g_t'(w)|^{5/48}(F-H)$. Note first of all that there exists a constant $M$ such that 
$$
|g_t'(w)|^{5/48}F\leq M|\phi_t'(w)|^{5/48},
$$ 
since the right-hand side is proportional to the probability that $w$ is connected to the boundary of the domain. A direct computation shows the same bound holds for $H$:
$$
|g_t'(w)|^{5/48}H\leq M|\phi_t'(w)|^{5/48}.
$$
Hence it is sufficient to prove that $|\phi_t'(w)|^{5/48}$ is uniformly integrable. Since this is monotone increasing in $t$, it is sufficient to show that $\E|\phi'_\tau(w)|^{5/48}<\infty$. This bound follows from known estimates on the dimension of $SLE_6$ (see e. g. \cite{RoSch}, lemma 6.3), concluding the proof of Theorem \ref{ThmF}.


\section{A Factorization Formula}
\label{Factorization}
In this section we prove the factorization formula obtained by Simmons, Kleban and Ziff in \cite{SKZ}. We introduce the following correlation functions:
$$
\begin{aligned}
&P_2(u_1,u_3):=\lim\limits_{u_2\rightarrow u_1}C(u_1,u_2,u_3)\left|\frac{u_2-u_1}{2}\right|^{-1/3}\\
&F(u_1,u_2,w):=\lim\limits_{r\rightarrow \infty}F(u_1,u_2,w,r)e^{5/48 r}\\
&P_3(u_1,w):=\lim\limits_{u_2\rightarrow u_1}F(u_1,u_2,w)\left|\frac{u_2-u_1}{2}\right|^{-1/3}\\
&P_{4}(u_1,u_3,w):=\lim\limits_{u_2\rightarrow u_1}C(u_1,u_2,u_3)F(u_1,u_2,u_3,w)\left|\frac{u_2-u_1}{2}\right|^{-1/3}
\end{aligned}
$$
\begin{proposition}
One has the following factorization formula:
\begin{equation}
\label{EqFact}
P^2_4(u_1,u_3,w)=K_F P_3(u_1,w)P_3(u_2,w)P_2(u_1,u_3), 
\end{equation}
where
$$
K_F=\frac{K_5^2K_3 2\, 2^{5/24}}{K_4^{2}\pi^{5/24}}=\frac{2\sqrt{\pi}}{\,_2 F^2_1\br{-\frac{1}{2},-\frac{1}{3},\frac{7}{6},1}\Gamma(\frac13)\Gamma(\frac76)}=\frac{2^7\pi^5}{3^{3/2}\Gamma(\frac13)^9}
$$
\end{proposition}
\begin{proof}
The proposition is proven by direct computation. One gets immediately from lemma \ref{lemma3pt} that 
\begin{equation}
\label{EqP2}
P_2(u_1,u_3)=K_3\frac{2^{1/3}}{(u_1-u_3)^{2/3}}
\end{equation}
In the notatin of lemma \ref{lemma4pt}, one has 
$$
\pi\omega\sim\frac{|u_1-u_2|\Im w}{|w-u_1|^2},\quad u_2\rightarrow u_1,
$$
and $|\phi'(w)|=\frac{1}{2\Im w}$, hence

\begin{equation}
\label{EqP3}
P_3(u_1,w)=\frac{K_4}{2^{5/48}}\frac{(\Im w)^{11/48}}{|u_1-w|^{2/3}}.
\end{equation}
It remains to figure out the limit of $F(u_1,u_2,u_3,w)$ as $u_2\rightarrow u_1$. Note that in this case the image $\psi_{u_1,u_2,u_3}(w)$ of $w$ under the mapping to the strip tends to infinity. We can write, as before,  $\psi_{u_1,u_2,u_3}=\psi_1\circ\psi_2$, where $\psi_2$ maps $\H, u_1,u_2,u_3$ to $\H, 0,1,\infty$ and 
$$
\psi_1^{-1}(z)=\frac{\cosh(\pi z)+1}{2}.
$$ 
Then we have 
$$
\begin{aligned}
\psi_2(w)\sim -\frac{(w-u_1)(u_3-u_1)}{(w-u_3)|u_2-u_1|}\\
\psi_2'(w)\sim \frac{(u_3-u_1)^2}{(w-u_3)^2|u_1-u_2|}.
\end{aligned}
$$
Writing $x=\Re{\psi_{u_1,u_2,u_3}(w)}$ and $y=\Im{\psi_{u_1,u_2,u_3}(w)}$, we obtain
$$
2\psi_2(w)\sim e^{\pi x}\frac{\cos(\pi y)+i\sin(\pi y)}{2},
$$
hence
$$
\begin{aligned}
&e^{\pi x} \sim |4\psi_2(w)|\sim \frac{4|w-u_1|(u_3-u_1)}{|w-u_3||u_2-u_1|},\\
&\pi y \sim \arg \psi_2(w)=\arg \frac{-w+u_1}{w-u_3}=\pi-\zeta,\\
&\psi_1'(\psi_2(w))\sim\frac{4}{\pi e^{\pi x}}\sim \frac{|w-u_3||u_2-u_1|}{\pi|w-u_1||u_3-u_1|}
\end{aligned}
$$
where $\zeta$ is the angle at $w$ in the triangle $(u_1,u_3,w)$. Plugging everything into (\ref{EqF}), we get that 
\begin{equation}
\label{EqP4}
\lim\limits_{u_2\rightarrow u_1}F(u_1,u_2,u_3,w)= \frac{K_5 2^{1/3}}{\pi^{5/48}}(\sin{\zeta})^{1/3}\Im w^{-5/48}
\end{equation}
The proposition now follows from (\ref{EqP2}), (\ref{EqP3}) and (\ref{EqP4}).
\end{proof}
\begin{rem}
The proposition has the following probabilistic interpretation. Let $P^{\varepsilon}_2(u_1,u_2)$, $P^{\varepsilon}_3(u,w)$ and  $P^{\varepsilon}_4(u_1,u_2,w)$ be the probabilities that $\varepsilon$-neighborhoods of corresponding points are connected by a percolation cluster (we understand $\varepsilon$-neighborhoods in the sense of $r(u_i)$ and $r(w)$, as defined in the introduction). Then (\ref{EqFact}) holds with $P_i$ replaced by $P_i^{\varepsilon}$ and equality replaced by equivalence as $\varepsilon \rightarrow 0$. Indeed, by definition $P_i^{\varepsilon}\sim P_i\varepsilon^{\sigma_i}$ with an appropriate $\sigma_i$. It is immediate to see that these power factors cancel out once plugged into (\ref{EqFact}).
\end{rem}
\begin{rem}
The constants $K_3,K_4,K_5$ in this paper are non-universal, in particular, they depend on our definition of neighborhoods of points (and in the lattice formulation they would depend on the lattice). The constant $K_F$, however, is conjectured to be universal, as all non-universal parts cancel out.
\end{rem}

\textbf{Acknowledgements.} We are grateful to Stanislav Smirnov for introducing us to the subject and many encouraging discussions, to Dmitry Chelkak for his explanations concerning the proof of Lemma \ref{lemma4pt}, to Kalle Kyt\"{o}l\"{a} for many useful discussions, and to Peter Kleban for useful discussion and proofreading the first version of the paper. Beliaev is partially supported by NSF grant DMS-0758492 and by Swiss National Science Foundation. Izyurov is partially supported by Swiss National Science Foundation, ERC AG CONFRA, EU RTN CODY and the Chebyshev Laboratory (Faculty of Mathematics and Mechanics, St Petersburg State University) under the grant 11.G34.31.0026 of the Government of the Russian Federation. 

\bibliography{sle} 
\bibliographystyle{abbrv}

\end{document}